\documentclass[11pt, reqno]{amsart}
\usepackage{amscd}
\usepackage{amsmath,calligra,mathrsfs}
\usepackage{amssymb}
\usepackage{latexsym}
\usepackage{amsthm}
\usepackage{graphicx}
\usepackage{arydshln} 
\usepackage[all]{xy}       

    \SelectTips{cm}{10}     

    \everyxy={<2.5em,0em>:} 

    \xyoption{web}          

                            \usepackage{amsmath,amsfonts,amsthm,verbatim,mathrsfs,amssymb}
\usepackage[usenames]{color}
\usepackage{tikz,graphicx,calligra}
\usetikzlibrary{decorations.shapes}
\usetikzlibrary{calc}
 \usetikzlibrary{arrows.meta}
\usepackage{esint}

\makeindex

\usepackage{tikz}
\usetikzlibrary{arrows}

\RequirePackage{color}
\definecolor{myred}{rgb}{0.75,0,0}
\definecolor{mygreen}{rgb}{0,0.5,0}
\definecolor{myblue}{rgb}{0,0,0.65}

\RequirePackage{ifpdf}
\ifpdf
  \IfFileExists{pdfsync.sty}{\RequirePackage{pdfsync}}{}
  \RequirePackage[pdftex,
   colorlinks = true,
   urlcolor = myblue, 
   citecolor = mygreen, 
   linkcolor = myred, 
   pagebackref,
   bookmarksopen=true]{hyperref}
\else
  \RequirePackage[hypertex]{hyperref}
\fi

\newcommand{\arxiv}[1]{\href{http://arxiv.org/abs/#1}{\tt arXiv:\nolinkurl{#1}}}
\newcommand{\arXiv}[1]{\href{http://arxiv.org/abs/#1}{\tt arXiv:\nolinkurl{#1}}}

\newtheorem{theorem}{Theorem}[section]
\newtheorem{lemma}[theorem]{Lemma}

\newtheorem{proposition}[theorem]{Proposition}

\theoremstyle{remark}
\newtheorem{remark}[theorem]{Remark}

\numberwithin{equation}{section}

\newcommand{\nc}{\newcommand}

\nc{\flags}{\mathcal{F}}

\nc{\KP}{\operatorname{KP}}

\nc{\re}{re}

\def\rank{\operatorname{rank}}
\def\im{\operatorname{im}}

\def\Q{\mathbb{Q}}

\def\C{\mathbb{C}}

\def\Z{\mathbb{Z}}
\def\F{\mathbb{F}}

\def\E{\mathcal{E}}

\def\A{\mathcal{A}}

\def\B{\mathcal{C}}

\def\T{\mathcal{T}}

\def\udots{\reflectbox{$\ddots$}}

\def\P{\mathcal{E}} 
\def\L{\mathcal{L}}

\def\Y{\widetilde{Y}}

\nc{\co}{\nabla}

\def\a{\alpha}

\def\b{\beta}

\def\la{\lambda}

\def\Hom{\operatorname{Hom}}
  \DeclareMathOperator{\sHom}{\mathscr{H}\text{\kern -3pt {\calligra\large om}}\,}

\def\uk{\underline{k}}

\def\Mat{\operatorname{Mat}}

\def\End{\operatorname{End}}

\def\id{\mbox{id}}

\newcommand{\map}[2]{\,{:}\,#1\!\longrightarrow\!#2}

\def\inv{^{-1}}

\numberwithin{equation}{section}

\title{Non-Perverse Parity Sheaves on the Flag Variety}

\author[Peter J McNamara]{Peter J McNamara}
\email{maths@petermc.net}

\date{\today}

\begin{document}

\begin{abstract}
We give examples of non-perverse parity sheaves on Schubert varieties for all primes.
\end{abstract}

\maketitle

\section{Introduction}

The notion of a parity sheaf was introduced in \cite{parity} and has since become an important object in modular geometric representation theory. An even/odd sheaf on a complex variety $X$ with coefficients in a field $k$ is an object of $D^b_c(X;k)$\footnote{$D^b_c(X;k)$ is the bounded derived category of constructible sheaves on $X$ with coefficients in $k$. We always work with the classical metric topology on $X$.} whose star and shriek restrictions to all points only have even/odd cohomology. A parity sheaf is a direct sum of an even and an odd sheaf. (We only consider parity sheaves for the zero pariversity in this paper).

In this paper, we take $X$ to be the variety of all complete flags in $\C^n$, and only consider sheaves which are constructible with respect to the stratification by Schubert varieties. Then by \cite[Theorem 4.6]{parity}, for each $w\in S_n$, there exists a unique indecomposable parity sheaf $\P_w$ whose support is the Schubert variety $X_w$, up to an overall homological shift. Up to homological shift, these constitute all indecomposable Borel-constructible parity sheaves on the flag variety $X$.
We normalise this shift such that when restricted to the Schubert cell, $\P_w$ is the constant sheaf shifted by $\dim(X_w)$. This ensures that when the characteristic of $k$ is zero, that $\E_w$ is isomorphic to the intersection cohomology sheaf $IC(X_w;k)$. We call such parity sheaves normalised.

Let $p$ be the characteristic of $k$.
We provide the first examples of normalised parity sheaves on Schubert varieties which are not perverse for primes $p>2$. Examples for $p=2$ were recently constructed in \cite{lw}. Our family of examples also includes parity sheaves which are arbitrarily non-perverse. We are not able to provide examples with $p$ greater than the Coxeter number, but we expect that such examples exist.
One geometric consequence is that
by Theorem \ref{uniqueness}, the non-perverseness of these sheaves proves that these Schubert varieties do not have any semi-small resolutions.
 Examples of non-perverse parity sheaves are also of representation-theoretic interest thanks to \cite{acharriche}.

Our examples generalise constructions of Kashiwara and Saito \cite{kashiwarasaito}, Polo (unpublished), and the author and Williamson \cite{tametorsion}.
We have phrased things in terms of parity sheaves as that is the more traditional formulation, although the construction of Theorem \ref{uniqueness} is more general. Despite this more general construction,  we rely on some of the theory of parity sheaves in our proof.

We would like to thank an anonymous referee whose diligent reading of this manuscript has resulted in a much improved exposition.

\section{Statement of the result}

Let $p$ be a prime. Let $d$ and $l$ be positive integers such that $p^d\geq l\geq 3$. Let $q=p^d$. 
Define the following permutation $y\in S_{q(l+2)}$:

\[
y(j)=\begin{cases}
(l+1)q &\text{if } j=1 \\
q+2-j &\text{if } 2\leq j\leq q \\
(l+2)q &\text{if } j=q+1 \\
(l+2)q+1-j &\text{if } q+2\leq j<(l+1)q \text{ and } j\not\equiv0,1\pmod q \\
(l+2)q-j &\text{if } q+2\leq j<(l+1)q \text{ and } j\equiv 0\pmod q \\
(l+2)q+2-j &\text{if } q+2\leq j<(l+1)q \text{ and } j\equiv 1\pmod q \\
1 &\text{if } j=(l+1)q \\
(2l+3)q-j &\text{if } (l+1)q<j<(l+2)q \\
q+1 &\text{if } j=(l+2)q.
\end{cases}
\]

Let $\P_y$ be the indecomposable parity sheaf supported on the Schubert variety $X_y$ with coefficients in $\F_p$, extending the constant sheaf shifted by $\dim(X_y)$. Our theorem is:

\begin{theorem}\label{main}
 \[
 \P_y\not\cong {^p}\tau_{\leq l-3}(\P_y).
 \]
\end{theorem}

Here ${^p}\tau_{\leq l-3}$ is the perverse truncation operator. Since $l\geq 3$, this implies that $\P_y$ is not perverse.

\section{Intersection Forms}


If $A$ is an indecomposable object in a Krull-Schmidt category and $X$ is any object, write $m(A,X)$ for the number of times $A$ appears as a direct summand of $X$.

Our main tool is the following result which computes the multiplicities of a direct summand via the rank of a bilinear form.

\begin{proposition}\cite[Proposition 3.2]{parity}\label{pairingprop}
Let $k$ be a local ring. Let $\pi\map{\widetilde{Y}}{Y}$ be a proper resolution of singularities. Let $y\in Y$, and suppose that the fibre $F=\pi\inv(y)$ is smooth. Write $i$ for the inclusion of $y$ in $Y$. Let $n$ be the dimension of $\Y$, $d$ the dimension of $F$ and $m$ be an integer. Let $B$ be the pairing
\begin{equation}\label{tool}
H^{2d-n-m}(F)\times H^{2d-n+m}(F) \to H^{2d}(F)
\end{equation}
given by $B(\a,\b)=\a\cup\b\cup e$, where $e$ is the Euler class of the normal bundle to $F$ in $Y$. Then
\[
m(i_*\uk[m],\pi_*\uk[n])=\rank(B).
\]
\end{proposition}

\begin{proof}
By general results about multiplicities of indecomposable objects in Krull-Schmidt categories, the multiplicity $m(i_*\uk[m],\pi_*\uk[n])$ is equal to the rank of the pairing
\begin{equation}\label{pairing}
\Hom(i_*\uk[m],\pi_*\uk_{\Y}[n])\times \Hom(\pi_*\uk_{\Y}[n],i_*\uk[m])\to \End(i_*\uk[m])\cong k.
\end{equation}

The following commutative diagram appears in the proof of \cite[Lemma 3.4]{parity} and arises through applying base change and adjunctions. The pairing $B'$ arises through composition with the canonical map from $\omega_F[-n]$ to $\uk_F[n]$ (which comes from applying the canonical natural transformation $i^!\to i^*$ to $\uk_{\Y}$), while the pairing (\ref{pairing}) appears along the bottom row.

\[
{
\xymatrix@C=0.3cm@R=0.5cm{
\Hom( \uk_{F}[m], \omega_F[-n]) \times
\Hom(\uk_F[n], \omega_F[m]) \ar[r]^(.66){B'} \ar[d]^\cong
& \Hom(\uk_F[m],\omega_F[m]) \ar[d]^\cong \\
\Hom(i_*\uk[m],\pi_*\uk_{\Y}[n])\times \Hom(\pi_*\uk_{\Y}[n],i_*\uk[m])
\ar[r] & \End(i_*
\uk[m])}
}\]

Since $F$ is smooth, $\omega_F\cong \uk_F[2d]$ and the canonical morphism from $\omega_F[-n]$ to $\uk_F[n]$ is identified with the Euler class $e$. Therefore the pairing (\ref{pairing}) is identified with the one stated in the proposition, completing the proof.
\end{proof}

\section{Geometry}\label{sec:geometry}

Let $x$ be the permutation
\[
 \left(\begin{array}{cccccc}
     J & 0 & \cdots & 0 & 0  \\
      0 & 0 & \cdots & J & 0\\
      \vdots & \vdots & \udots & \vdots & \vdots  \\
      0 & J & \cdots & 0 & 0  \\
      0 & 0 & \cdots & 0 & J  \\
      \end{array} \right)
\]
where $J$ is the $q\times q$ antidiagonal matrix.

We compute the slice to the Schubert variety $X_x$ in $X_y$ through the point $x$. Using the techniques discussed in the proof of \cite[Proposition 3.2]{cell}, this is given by matrices of the form

\[
 \left(\begin{array}{cccccc}
     J & 0 & 0 & \cdots & 0 & 0 \\
      A_l  & 0 & 0 & \cdots & J & 0 \\
      \vdots & \vdots & \vdots & \udots & \vdots & \vdots \\
      A_2 & 0 & J & \cdots & 0 & 0 \\
      A_1 & J & 0 & \cdots & 0 & 0 \\
      0 & B_1 & B_2 & \cdots & B_l & J
      \end{array} \right)
\]

subject to the conditions
\begin{equation}\label{y}
 B_i J A_i=0, \quad \rank\begin{pmatrix}
 A_l \\
 \vdots \\
 A_1
 \end{pmatrix}\leq 1, \quad \rank\begin{pmatrix}
 B_1 & \cdots & B_l
 \end{pmatrix}\leq 1
\end{equation}
for all $i$. Here, each $A_i$ and $B_i$ is a $q\times q$ matrix.
Let $Y$ be this slice. So $Y$ is the space of all $(2l)$-tuples of $q\times q$ matrices $A_1,\ldots,A_l,B_1,\ldots B_l$ subject to the conditions (\ref{y}) above.

Let \begin{equation*}\label{defn:ytilde}
 \begin{split}
 \Y=\{ (h,\ell_1,\ldots,\ell_l,\ell,A_1,\ldots,A_l,C_1,\ldots,C_l)\mid h\in Gr(q-1,q), \\ \ell_1,\dots,\ell_l,\ell\in Gr(1,q), A_i\in \Hom(\C^q/h,\ell_i),B_i\in \Hom(\C^q/\ell_i,\ell) \}  
 \end{split}
\end{equation*}
and $\pi\map{\Y}{Y}$ be defined by
\begin{equation}\label{defn:pi}
\pi(h,\ell_1,\ldots,\ell_l,\ell,A_1,\ldots,A_l,C_1,\ldots,C_l)=(A_1,\ldots,A_l,C_1J,\ldots,C_lJ).
\end{equation}

This is a resolution of singularities of $Y$ and $\Y$ is the total space of a vector bundle $\E$ on $Z:=(\mathbb{P}^{q-1})^{l+2}$. 

We have
\[
 H^*(Z) \cong \Z[w]/(w^q)\otimes ( \bigotimes_{i=1}^l \Z[a_i]/(a_i^q) )\otimes \Z[z]/(z^q)
\]
where the variables $w$ and $z$ come from the factors of $Z$ corresponding to the choices of $h$ and $\ell$ in $\Y$ respectively, while the $a_i$ come from the choice of $\ell_i$.

\begin{lemma}
The Euler class of $\E$ is given by
\begin{equation}\label{eofe}
 e(\E) =\prod_{i=1}^l \left((a_i+w)\sum_{j=0}^{q-1}a_i^j z^{q-1-j}\right).
\end{equation}
\end{lemma}

\begin{proof}
The bundle $\E$ is naturally a direct sum of $2l$ vector bundles
\[
\E\cong \bigoplus_{i=1}^l \A_i \oplus \bigoplus_{i=1}^l \mathcal{C}_i,
\]
where the bundles $\A_i$ and $\mathcal{C}_i$ correspond to the choice of $A_i$ and $C_i$ in $\Y$. The bundle $\A_i$ is a line bundle with first Chern class $a_i+w$. 

On $\mathbb{P}^{q-1}$, write $\L$ for the tautological line bundle, and $\T$ for the trivial vector bundle of rank $q$. Then, restricted to the appropriate $\mathbb{P}^{q-1}\times \mathbb{P}^{q-1}$, we have
\[
\B_i\cong  p_1^*(\T/\L)^* \otimes p_2^* \L.
\]

The total Chern class of $(\T/\L)^*$ is given by
\[
c((\T/\L)^*)=\sum_{j=1}^\infty a_i^j.
\]

The computation of $e(\mathcal{C}_i)$ proceeds via (\ref{chern}) below, whose proof follows easily using the splitting principle:

Let $V$ be a rank $n$ vector bundle and $L$ be a line bundle. Then
\begin{equation}\label{chern}
e(V\otimes L)=\sum_{i=0}^n c_i(V) c_1(L)^{n-i}.
\end{equation}

Using this, we obtain
\[
e(\B_i)=\sum_{j=0}^{q-1}a_i^j z^{q-1-j}.
\]

Together with the formula for the Euler class (=first Chern class) of $\A_i$, this implies the formula in the lemma.
\end{proof}

Motivated by Proposition \ref{pairingprop}, we consider the pairing
\[
 \langle \cdot,\cdot\rangle \map{H^{2(q-l)}(Z)\times H^{2(q-2)}(Z)}{H^{2(l+2)(q-1)}(Z)}
\]
given by
\[
 \langle \sigma,\tau\rangle = \sigma \cup \tau \cup e(\E).
\]

The ring $H^*(Z)$ has a basis consisting of monomials in the variables $w,a_i,z$, where the exponent of each variable is less than $q$. 
When computing the intersection pairing $\langle \cdot ,\cdot\rangle$ between two of these monomials in $H^*(Z)$, one notices that it only depends on the power of $w$ in their product.
If this power is $j$, then the pairing takes the value ${l \choose q-1-j}$, for this is the number of ways of choosing terms in the product formula (\ref{eofe}) for $e(\E)$ which produce the right power of $w$, and once these choices are made, the rest are uniquely determined.

Deleting duplicate rows that are irrelevant for computing the rank, the matrix for our intersection form with respect to the monomial basis becomes

\begin{equation}\label{matrix}
M=
 \left(\begin{array}{cccccccccccc}
    l & {l \choose 2} & {l \choose 3} & \cdots & & l & 1 & & & & \\
    1 & l & {l \choose 2} & {l \choose 3} & \cdots  & & l & 1 & & & \\
    & 1 & l & {l \choose 2} & & & & \ddots & \ddots & & \\
    & & \ddots & \ddots & \ddots & & & & l & 1 & \\
    & & & 1 & l & {l \choose 2} & \cdots & && l & 1 \\
    & & & & 1 & l & {l \choose 2} & \cdots && & l
    \end{array} \right)
\end{equation}
%

This matrix has $q-1$ columns, corresponding to the powers of $w$ between $0$ and $q-2$ inclusive that appear in $H^{2(q-2)}(Z)$ and $q-l+1$ rows, corresponding to the powers of $w$ between $0$ and $q-l$ inclusive that appear in $H^{2(q-l)}(Z)$.

\section{Reduction to the slice}

This section is likely to be of independent interest. We require the following result about direct summands of a pushforward sheaf. The same proof works when $k$ is a local ring, but for simplicity we assume $k$ is a field.

\begin{theorem}\label{uniqueness}
Let $X$ be an irreducible complex analytic space. 
Let $k$ be a field. Then there exists a unique indecomposable object $\E(X;k)\in D^b_c(X;k)$ which restricts to the constant sheaf shifted by $\dim X$ on an open subset of $X$, and is a direct summand of $\sigma_*\uk_Y[\dim X]$ for any proper resolution of singularities $\sigma:Y\to X$.
\end{theorem}

\begin{proof}
Let $\sigma:Y\to X$ and $\pi:Z\to X$ be two proper resolutions of singularities. 
Consider
\(
\Hom(\sigma_*\uk_Y,\pi_*\uk_Z)\cong H^{BM}_{2d}(Y\times_X Z)
\), where $H^{BM}_*$ stands for Borel-Moore homology and $d=\dim_\C(X)$.
Let $U$ be a connected open dense subset of $X$ such that $\sigma$ and $\pi$ are isomorphisms over $U$. Write $j:U\to X$ for the inclusion. 
As $\sigma$ and $\pi$ are isomorphisms over $U$, there is a canonical inclusion of $U$ into $Y\times_X Z$. The closure $\overline U$ of $U$ in $Y\times_XZ$ is an irreducible component of $Y\times_XZ$, hence defines a fundamental class\footnote{The definition of fundamental class is the same as in the algebraic case \cite[\S 2.6.12]{chrissginzburg}, using the fact that the singular locus is of real codimension at least two.} $\a:=[\overline{U}]\in H^{BM}_{2d}(Y\times_XZ)\cong \Hom(\sigma_*\uk_Y,\pi_*\uk_Z)$. This class satisfies $j^*\a=\id\in \Hom(\uk_U,\uk_U)$.

Similarly we define $\b\in \Hom(\pi_*\uk_Z,\sigma_*\uk_Y)$ with $j^*\b=\id\in\Hom(\uk_U,\uk_U)$.

The category $D^b_c(X;k)$ is Krull-Schmidt. Therefore, since $\uk_U$ is indecomposable and $j^*(\sigma_*\uk_Y)\cong \uk_U$, we can write $\sigma_*\uk_Y\cong A\oplus C$ where $A$ is indecomposable and $j^*C=0$. Similarly we have $\pi_*\uk_Z\cong B\oplus D$ with $B$ indecomposable and $j^*D=0$.
 Using these direct sum decompositions, consider the components of $\a$ and $\b$. They induce morphisms which by abuse of notation we call $\a:A\to B$ and $\b:B\to A$.

Consider $\beta\alpha\in \End(A)$. Under the ring homomorphism $j^*:\End(A)\to \End(\uk_U)\cong \uk$, we have $j^*(\b\a)=1$. As $A$ is indecomposable, $\End(A)$ is local. This implies that $\b\a$ is a unit in $\End(A)$. Similarly $\a\b$ is a unit in $\End(B)$. 
We thus have two indecomposable objects $A$ and $B$, with morphisms between them $\a$ and $\b$, whose compositions in each direction are isomorphisms. This is enough to conclude $A\cong B$. This completes the proof, with $\E(X;k)=A[\dim X]$. Note that $\E(X;k)$ always exists because a resolution of singularities always exists.
\end{proof}

\begin{remark}
For a Schubert variety $X_y$, the sheaf $\E(X_y;\F_p)$ is the same as the parity sheaf $\E_y$, by considering a Bott-Samelson resolution of $X_y$ \cite[Prop 4.11]{parity}.
\end{remark}

Now we show how to use this result to restrict our attention to slices. Let $X$ be an irreducible complex analytic space. Let $Z\subset X$ be a closed analytic subset. Let $Y$ be a slice to $Z$ in $X$. This implies that there exists an open subset $U\subset X$ and smooth $V$ such that $U\cong Y\times V$ and under this isomorphism the copy of $Y$ in $X$ gets sent to $Y\times \{v\}$ for some $v\in V$.

Let $Y'\to Y$ be a resolution of $Y$. Then $Y'\times V\to U$ is a resolution of $U$. The sheaf $\E(Y;k)\boxtimes\uk_V[\dim V]$ on $U$ is an indecomposable direct summand of the pushforward of the shifted constant sheaf under this resolution which is generically constant, hence by the above theorem is isomorphic to $\E(U;k)$.

Let $j\map{U}{X}$ denote the inclusion. Pulling back a resolution of $X$ to $U$ via $j$ shows that $\E(U;k)$ is a direct summand of $j^\ast \E(X;k)$.

Therefore $\E(Y;k)\boxtimes \uk_V[\dim V]$ is a direct summand of $j^\ast \E(X;k)$. So if we can show that $\E(Y;k)\not\cong{}^p\tau_{\leq l-3}(\E(Y;k))$, then that will imply that $\E(X;k)\not\cong{}^p\tau_{\leq l-3}(\E(X;k))$. This is the method by which we can restrict our attention to the slice.

\section{Fin}

\begin{lemma}\label{rank}
 The matrix $M$ from (\ref{matrix}) has rank $q-l+1$ over $\Q$ and $q-l$ over $\F_p$.
\end{lemma}

\begin{proof}
 Identify the rows of $M$ with the sequence of polynomials
 \begin{align*}
  &(1+x)^l-1 \\
  x&(1+x)^l \\
  x^2&(1+x)^l \\
  & \vdots \\
  x^{q-l-1}&(1+x)^l \\
  x^{q-l}&(1+x)^l -x^q.
 \end{align*} 
If there is a linear dependence, then $A+Bx^q$ is divisible by $(1+x)^l$ for some constants $A$ and $B$. Over $\Q$, this is impossible since $A+Bx^q$ has distinct roots over $\C$. This computes the rank over $\Q$.

Over $\F_p$, the polynomial $(1+x)^l$ divides $x^q+1$, which easily leads to a linear dependence amongst the rows of $M$. It is obvious that the rank is at least $q-l$, completing the proof in this case.
\end{proof}

Recall the resolution $\pi\map{\Y}{Y}$ defined in \S \ref{sec:geometry}. Let $Z=\Y\times_Y\Y.$

\begin{lemma}\label{lem:even}
The Borel-Moore homology groups $H_i^{BM}(Z;\Z)$ are free over $\Z$ and vanish when $i$ is odd.
\end{lemma}

\begin{proof}

The variety $Z$ is
\begin{multline*}
Z=\{(A_1,\ldots,A_l,B_1,\ldots,B_l,h,h',\ell_1,\ldots,\ell_l,\ell_1'\ldots\ell'_l,\ell,\ell')\mid 
 \\
 A_i,B_i\in \Hom(\C^q,\C^q);\  h,h'\in Gr(q-1,q);\  \ell_1,\ldots,\ell_l,\ell_1'\ldots\ell'_l,\ell,\ell'\in Gr(1,q);
 \\ h,h'\subset \ker(A_i);\   \im(A_i)\subset \ell_i,\ell_i'\subset \ker B_i;\  \im(B_i)\subset \ell,\ell'
\}.
\end{multline*}

We now construct a stratification of $Z$. For each $I\subset \{1,2,\ldots,l\}\sqcup\{s,t\}$, we define a stratum $Z_I$ consisting of tuples $(A_1,\ldots,A_l,B_1,\ldots,B_l,h,h',\ell_1,\ldots,\ell_l,\ell_1'\ldots\ell'_l,\ell,\ell')\in Z$ subject to the conditions $\ell_i=\ell_i'$ if $i\in I$, $h=h'$ if $s\in I$, $\ell=\ell'$ if $t\in I$, $\ell_i\neq\ell_i'$ if $i\notin I$, $h\neq h'$ if $s\notin I$ and $\ell\neq\ell'$ if $t\notin I$.

Each stratum $Z_I$ is a vector bundle over a product of spaces that are either $\mathbb{P}^{q-1}$ or $(\mathbb{P}^{q-1}\times\mathbb{P}^{q-1})\setminus\Delta$, where $\Delta$ is the diagonal. Therefore $H_i^{BM}(Z_I;\Z)$ is free over $\Z$ and vanishes when $i$ is odd. Since the $Z_I$ stratify $Z$, the same is true for the Borel-Moore homology of $Z$.
\end{proof}

Reimagine $Y$ as a subspace of the space of representations of the following quiver $Q$, where there are $l$ vertices in the central column.
%
\begin{center}
\begin{tikzpicture}[scale=1]
\tikzset{dots/.style 2 args={decorate,decoration=
         {shape backgrounds,shape=circle,shape size=#1,shape sep=#2}}} 
\pgfmathsetmacro{\radius}{0.06} 
\pgfmathsetmacro{\width}{2.5}   
\pgfmathsetmacro{\height}{1}    
\pgfmathsetmacro{\gap}{0.4}     
\pgfmathsetmacro{\arrowsize}{1.1mm} 

\draw[-{Latex[length=\arrowsize]}] (-\width+0.15,0.05)--(-0.15,\height-0.02);
\draw[-{Latex[length=\arrowsize]}] (-\width+0.15,0.01)--(-0.15,\height-\gap);
\draw[-{Latex[length=\arrowsize]}] (-\width+0.15,-0.05)--(-0.15,-\height+0.02);
\draw[-{Latex[length=\arrowsize]}] (0.15,\height-0.02)--(\width-0.15,0.05);
\draw[-{Latex[length=\arrowsize]}] (0.15,\height-\gap)--(\width-0.15,0.01);
\draw[-{Latex[length=\arrowsize]}] (0.15,-\height+0.02)--(\width-0.15,-0.05);
\draw[fill=blue] (-\width,0) circle (\radius);
\draw[fill=blue] (\width,0) circle (\radius);
\draw[fill=blue] (0,\height) circle (\radius);
\draw[fill=blue] (0,\height-\gap) circle (\radius);
\draw[fill=blue] (0,-\height) circle (\radius);
\draw[dots={0.2mm}{1.2mm},fill] (0,\height-2*\gap)--(0,\gap-\height); 
\end{tikzpicture}
\end{center}
The representations have dimension vector
\begin{center}
\begin{tikzpicture}[scale=1]
\tikzset{dots/.style 2 args={decorate,decoration=
         {shape backgrounds,shape=circle,shape size=#1,shape sep=#2}}} 
\pgfmathsetmacro{\radius}{0.06} 
\pgfmathsetmacro{\width}{2.5}   
\pgfmathsetmacro{\height}{1}    
\pgfmathsetmacro{\gap}{0.4}     
\pgfmathsetmacro{\arrowsize}{1.1mm} 

\draw[-{Latex[length=\arrowsize]}] (-\width+0.15,0.05)--(-0.15,\height-0.02);
\draw[-{Latex[length=\arrowsize]}] (-\width+0.15,0.01)--(-0.15,\height-\gap);
\draw[-{Latex[length=\arrowsize]}] (-\width+0.15,-0.05)--(-0.15,-\height+0.02);
\draw[-{Latex[length=\arrowsize]}] (0.15,\height-0.02)--(\width-0.15,0.05);
\draw[-{Latex[length=\arrowsize]}] (0.15,\height-\gap)--(\width-0.15,0.01);
\draw[-{Latex[length=\arrowsize]}] (0.15,-\height+0.02)--(\width-0.15,-0.05);
\node at (-\width,0) {$q$};
\node at  (\width,0) {$q$};
\node at (0,\height) {$q$};
\node at (0,\height-\gap) {$q$};
\node at (0,-\height) {$q$};
\draw[dots={0.2mm}{1.2mm},fill] (0,\height-2*\gap)--(0,\gap-\height); 
\end{tikzpicture}
\end{center}
The group $G=GL_q(\C)^{l+2}$ acts on $Y$ in a manner such that $[Y/G]$ is a substack of the moduli stack of representations of $Q$. 

\begin{lemma}\label{yparity}
The map $\pi\map{\Y}{Y}$ defined in (\ref{defn:pi}) is a proper even\footnote{The definition of an even morphism is given in \cite[Definition 2.33]{parity}. A sufficient condition that implies evenness is that each fibre is equivariantly simply connected with no odd cohomology.} $G$-equivariant resolution of singularities, stratified with respect to a stratification $Y=\sqcup Y_\la$ such that $H^*_G(Y_\la;\mathcal{L})$ vanishes in odd degrees for all $G$-equivariant local systems $\mathcal{L}$ on $Y_\la$.
\end{lemma}

\begin{remark}
This lemma allows us to have access to the parity sheaf machinery from \cite{parity} for $G$-equivariant sheaves on $Y$. 
\end{remark}
\begin{proof}
Let $I$ and $J$ be two subsets of $\{1,2,\ldots,l\}$. Let
\begin{equation}\label{def:yij}
Y_{I,J}=\{(A_1,\ldots,A_l,B_1,\ldots,B_l)\in Y\mid A_i=0 \iff i\in I,\  B_j=0\iff j\in J\}.
\end{equation}
This is the stratification of $Y$ into $G$-orbits. We will show that this gives the desired stratification of $Y$. In the moduli interpretation of $Y$, inside each $G$-orbit there is a unique representation $M_{IJ}$ of $Q$, up to isomorphism.

The module $M_{IJ}$ decomposes as a direct sum $M_{IJ}=X_I\oplus Y_J\oplus Z$, where $X_I$ is an indecomposable with dimension 1 at the leftmost vertex, $Y_J$ is an indecomposable with dimension 1 at the rightmost vertex, and $Z$ is a direct sum of simple modules, except when $I=\{1,2,\ldots ,l\}$ or $J=\{1,2,\ldots,l\}$, when $X_I$ or $Y_J$ respectively do not appear in the decomposition.

Put an order on these indecomposables where the simple at the leftmost vertex comes earliest in the order, then $X_I$, then the simples at the middle vertices, then $Y_J$, then the simple at the rightmost vertex. With this ordering, we can decompose $M_{IJ}$ into indecomposables, $M_{IJ}=\oplus_{i=1}^m M_i^{\oplus n_i}$ where $\Hom_Q(M_i,M_j)=0$ if $i<j$.

Therefore $\End_Q(M_{IJ})^\times$ surjects onto $\prod_i \Mat_{n_i}(\End_Q(M_i))^\times$ with unipotent kernel. Each $\End_Q(M_i)$ is isomorphic to $\C$.

The quotient stack $[Y_{IJ}/G]$ is isomorphic to $[\mathrm{pt}/\End_Q(M_{IJ})^\times]$. This shows that $H^*_G(Y_{IJ};\mathcal{L})$ is a free $k$-module and vanishes in odd degrees since these properties hold for the stack $[\mathrm{pt}/GL_n(\C)]$.

It is clear that $\pi$ is a proper $G$-equivariant resolution of singularities, and thus is stratified for the stratification into $G$-orbits. It is even because every fibre of $\pi$ is a product of projective spaces.
\end{proof}

We now come to the proof of Theorem \ref{main}. \begin{proof}
Let $n=(q-1)(l+2)+ql$ be the common dimension of $Y$ and $\widetilde{Y}$. 
Decompose $\pi_*\underline{\Z}_p[n]$ into indecomposables
\[
\pi_*\underline{\Z}_p[n]\cong \bigoplus_t \P_t^{n_t}.
\]
The endomorphism ring of $\pi_*\underline{\Z}_p[n]$ is $\End(\pi_*\underline{\Z}_p[n])\cong H_{2n}^{BM}(Z;\Z_p)$, which governs the decomposition into irreducibles.

The spectral sequence with $E_2^{p,q}=H^p(BG;H_{-q}^{BM}(Z))$ converges to the $G$-equivariant Borel-Moore homology $H_*^G(Z)$. This spectral sequence is concentrated in even degrees by Lemma \ref{lem:even} and hence degenerates at the $E_2$ page. Therefore $H_{2n}^G(Z)$ surjects onto $H_{2n}^{BM}(Z)$.
As a consequence, each $\E_t$ is the deequivariantisation of an indecomposable $G$-equivariant sheaf, which is an equivariant parity sheaf by Lemma \ref{yparity}.

Again by Lemma \ref{lem:even}, the convolution algebra $H_{2n}^{BM}(Z;\Z_p)$ surjects onto $H_{2n}^{BM}(Z;\F_p)$. Therefore 
each $\P_t\otimes \F_p$ is also indecomposable.

Let $i$ be the inclusion of $\{0\}$ in $Y$.
By Proposition \ref{pairing} and Lemma \ref{rank}, the multiplicity of $i_*\uk[l-2]$ in $\pi_*\uk[n]$ is $q-l+1$ when $k=\Q_p$ and $q-l$ when $k=\F_p$. Therefore there exists a unique $t$ such that
\begin{equation}\label{multiplicity}
m(i_*\underline{\Q}_p[l-2],\P_t\otimes \Q_p)=1
\end{equation}
and $\E_t$ is not a skyscraper sheaf at 0.

This setup is symmetric under the action of the symmetric group $S_l$ which permutes the indices $\{1,2,\ldots,l\}$ of elements of $Y$. So by uniqueness of $t$, $\P_t$ is $S_l$-invariant.
From the classification of equivariant parity sheaves in \cite{parity}, $\E_t$ is up to homological shift the parity extension of a constant sheaf on a $S_l$-invariant $G$-orbit on $Y$.

%


There are only two $S_l$-invariant intermediate $G$-orbits in $Y$. They are $Y_{\{1,2,\ldots,l\},\emptyset}$ and $Y_{\emptyset,\{1,2,\ldots,l\}}$ in the notation of (\ref{def:yij}). 

Let $W$ be the closure of $Y_{\{1,2,\ldots,l\},\emptyset}$.
Let us assume that $\E_t$ has support $W$. 
Define
\[
\widetilde{W}=\{(\ell,B_1,B_2,\ldots B_l)\mid \ell\in Gr(1,q), B_i\in \Hom(\C^q,\ell)\}.
\]

Write $\sigma$ for the map from $\widetilde{W}$ to ${W}$. This is a $G$-equivariant even resolution of singularities. Therefore, in the $G$-equivariant derived category, $\E_t$ must appear as a direct summand of $\sigma_*\underline{\Z}_p$ (up to homological shift) as it is the unique indecomposable $G$-equivariant parity sheaf extending the constant sheaf and thus the same statement holds when we forget the equivariant structure.

The space $\widetilde{W}$ is the total space of a vector bundle $\mathcal{F}\cong \mathcal{O}(1)^{\oplus lq}$ over $\mathbb{P}^{q-1}$, where the zero section is the fibre over $0\in Y$. The Euler class $e(\mathcal{F})$ is a power of the Chern class $c_1(\mathcal{O}(1))$.
In computing the intersection form (\ref{tool}) for the resolution $\sigma$, once irrelevant rows are deleted, one is left with the antidiagonal matrix $J$, which has the same rank over $\Q$ and $\F_p$. 

Since the rank doesn't change, we deduce that
\(
m(i_*\underline{\Q} _p[l-2],\P_t\otimes \Q_p)=0
\), contradicting (\ref{multiplicity}). Therefore $\E_t$ cannot have support equal to $W$.

Similarly $\E_t$ cannot have support equal to the closure of the other $S_l$-equivariant intermediate stratum $Y_{\emptyset,\{1,2,\ldots,l\}}$.

Therefore $\P_t$ is an extension of the shifted constant sheaf on the open stratum in $Y$. As the stalk of $\P_t$ at 0 is free over $\Z_p$ and satisfies a parity vanishing property, it is nonzero in degree $l-2$ using (\ref{multiplicity}). Therefore $\P_t\otimes \F_p$ has nonzero stalk cohomology at 0 in degree $l-2$. Thus
\[
\P_t\otimes \F_p \not\cong 
{^p\tau}_{\leq l-3}(\P_t\otimes \F_p).
\]
Since $\E_t\otimes \F_p$ is an indecomposable extension of the shifted constant sheaf on the open stratum on $Y$ and occurs as a summand of $\pi_*\underline{\F_p}[n]$, by Theorem \ref{uniqueness}, $\E_t\otimes \F_p\cong \E(Y;\F_p)$. By the reduction to the slice argument made in the previous section, this completes the proof of Theorem \ref{main}.
\end{proof}

\end{document}